\documentclass[12 pt]{amsart}

\usepackage{hyperref}
\usepackage{etex}
\usepackage[shortlabels]{enumitem}
\usepackage{amsmath}
\usepackage{amsxtra}
\usepackage{amscd}
\usepackage{amsthm}
\usepackage{amsfonts}
\usepackage{amssymb}
\usepackage{eucal}
\usepackage[all]{xy}
\usepackage{graphicx}
\usepackage{tikz-cd}
\usepackage{mathrsfs}
\usepackage{subfiles}
\usepackage{mathpazo}
\usepackage[colorinlistoftodos, textsize=tiny]{todonotes}
\usepackage{morefloats}
\usepackage{pdfpages}
\usepackage{thm-restate}
\usepackage[utf8]{inputenc}
\usepackage{epigraph}
\usepackage{csquotes}
\usepackage[margin=1.5in]{geometry}
\usepackage{adjustbox}
\usepackage{scalerel}
\usepackage{stackengine}
\stackMath
\newcommand\reallywidehat[1]{%
\savestack{\tmpbox}{\stretchto{%
  \scaleto{%
    \scalerel*[\widthof{\ensuremath{#1}}]{\kern-.6pt\bigwedge\kern-.6pt}%
    {\rule[-\textheight/2]{1ex}{\textheight}}
  }{\textheight}%
}{0.5ex}}%
\stackon[1pt]{#1}{\tmpbox}%
}
\graphicspath{ {images/} }

\RequirePackage{color}
\definecolor{myred}{rgb}{0.75,0,0}
\definecolor{mygreen}{rgb}{0,0.5,0}
\definecolor{myblue}{rgb}{0,0,0.65}

\usepackage{color}

\usepackage{hyperref}
\hypersetup{citecolor=blue}
\usepackage{tikz}
\usetikzlibrary{matrix,arrows,decorations.pathmorphing}

\theoremstyle{plain}
\newtheorem{theorem}{Theorem}
\newtheorem{slogan}[theorem]{Slogan}
\newtheorem{proposition}[theorem]{Proposition}
\newtheorem{lemma}[theorem]{Lemma}
\newtheorem{corollary}[theorem]{Corollary}

\newtheorem{problem}[theorem]{Problem}
\theoremstyle{definition}
\newtheorem{definition}[theorem]{Definition}
\newtheorem{remark}[theorem]{Remark}
\newtheorem{example}[theorem]{Example}

\newtheorem{question}[theorem]{Question}
\newtheorem{conjecture}[theorem]{Conjecture}

\theoremstyle{remark}
\newtheorem{notation}[theorem]{Notation}

\numberwithin{equation}{section}
\numberwithin{theorem}{section}
\numberwithin{slogan}{section}
\numberwithin{proposition}{section}
\numberwithin{lemma}{section}
\numberwithin{corollary}{section}
\numberwithin{situation}{section}
\numberwithin{problem}{section}
\numberwithin{definition}{section}
\numberwithin{remark}{section}
\numberwithin{example}{section}
\numberwithin{exercise}{section}
\numberwithin{counterexample}{section}
\numberwithin{convention}{section}
\numberwithin{question}{section}
\numberwithin{conjecture}{section}
\numberwithin{goal}{section}
\numberwithin{warn}{section}
\numberwithin{fact}{section}
\numberwithin{notation}{section}
\numberwithin{construction}{section}

\newcommand\nc{\newcommand}
\nc\on{\operatorname}
\nc\renc{\renewcommand}

\newcommand\bp{\mathbb P}

\newcommand*{\shom}{\mathscr{H}\kern -.5pt om}
\newcommand*{\stor}{\mathscr{T}\kern -.5pt or}
\newcommand*{\sext}{\mathscr{E}\kern -.5pt xt}

\makeatletter
\providecommand\@dotsep{5}
\renewcommand{\listoftodos}[1][\@todonotes@todolistname]{%
\@starttoc{tdo}{#1}}
\makeatother

\makeatletter
\newcommand{\customlabel}[2]{\protected@write \@auxout {}{\string \newlabel {#1}{{#2}{\thepage}{#2}{#1}{}} }\hypertarget{#1}{#2}}

\renewcommand\hom{\mathrm{Hom}}

\DeclareMathOperator\rk{rk}

\DeclareMathOperator\pic{Pic}

\DeclareMathOperator\im{im}

\DeclareMathOperator\sym{Sym}

\DeclareMathOperator\Mod{Mod}
\DeclareMathOperator\HMod{HMod}

\DeclareMathOperator\PW{PW}
\DeclareMathOperator\HPW{HPW}

\DeclareFontFamily{U}{wncy}{}
\DeclareFontShape{U}{wncy}{m}{n}{<->wncyr10}{}
\DeclareSymbolFont{mcy}{U}{wncy}{m}{n}
\DeclareMathSymbol{\Sha}{\mathord}{mcy}{"58}

\def\listtodoname{List of Todos}
\def\listoftodos{\@starttoc{tdo}\listtodoname}

\title[The algebraic geometry of the Putman-Wieland conjecture]{An introduction to the algebraic geometry of the Putman-Wieland conjecture}
\author{Aaron Landesman}
\author{Daniel Litt}

\AtEndDocument{\bigskip{\footnotesize%
  \textsc{Department of Mathematics, MIT, 
182 Memorial Dr,   
\mbox{Cambridge, MA 02142}} \par
Aaron Landesman: \texttt{aaronl@mit.edu} \par
    \addvspace{\medskipamount}
  \textsc{Department of Mathematics, University of Toronto, 
	  Bahen Centre,
40 St. George St.,
  \mbox{Toronto, Ontario, Canada, M5S 2E4}}  \par
  Daniel Litt: \texttt{daniel.litt@utoronto.ca} \par
  }}

\usepackage{microtype}
\begin{document}

\begin{abstract}
We give algebraic and geometric perspectives on our prior results toward the
Putman-Wieland conjecture.
This leads to interesting new constructions of families of ``origami" curves whose Jacobians have high-dimensional isotrivial isogeny factors. We also explain how a hyperelliptic analogue of the Putman-Wieland conjecture fails, following work of Markovi\'{c}.
\end{abstract}

\maketitle

\section{Introduction}

The goal of this paper is to expand on some of our recent results towards the
Putman-Wieland conjecture. Our hope is that this 
somewhat leisurely
exposition will serve as a useful entry point for geometric topologists hoping to use the 
Hodge-theoretic techniques developed in \cite{landesmanL:canonical-representations} to study mapping class groups. While many of the results of that paper are somewhat stronger than those explained here, our hope is that stripping away some of the technical aspects of the proofs therein will clarify the relevant arguments.

\subsection{Review of the Putman-Wieland conjecture}
\label{subsection:putman-wieland-statement}

We use $\Sigma_{g,n}$ to denote an orientable topological surface of genus $g$
with $n$ punctures.
Given a finite unramified $H$-cover of punctured topological surfaces $\Sigma_{g',n'} \to \Sigma_{g,n}$,
there is an action of a finite index subgroup $\Gamma$ of the mapping class group $\Mod_{g,n+1}$ on
$H_1(\Sigma_{g'}, \mathbb C)$, as we now explain.
Explicitly, $\on{Mod}_{g,n+1}$ acts on $\pi_1(\Sigma_{g,n},x)$ for some basepoint $x$, and we can take $\Gamma$ to be the stabilizer of the surjection
$\phi: \pi_1(\Sigma_{g,n},x) \twoheadrightarrow H$, where 
$\phi$ corresponds to the cover $\Sigma_{g',n'} \to \Sigma_{g,n}$.
Then, for $x' \in \Sigma_{g',n'}$ mapping to $x$, $\Gamma$ acts on
$\pi_1(\Sigma_{g',n'},x')$ and hence on $H_1(\Sigma_{g'},\mathbb C)$.

\begin{conjecture}[Putman-Wieland, \protect{\cite[Conjecture 1.2]{putmanW:abelian-quotients}}]
\label{conjecture:putman-wieland-intro}
Fix $g \geq 2, n \geq 0$.
For any unramified cover 
$\Sigma_{g',n'} \to \Sigma_{g,n}$,
the vector space $H_1(\Sigma_{g'}, \mathbb C)$ has no 
nonzero vectors with finite orbit under the action of $\Gamma$.
\end{conjecture}

Note that the Putman-Wieland conjecture is false when $g=2$ (see \cite{markovic} or \autoref{proposition:hyperelliptic-counterexample} below). We state an equivalent version of the Putman-Wieland conjecture in
\autoref{conjecture:putman-wieland}.

\subsection{Our prior results toward the Putman-Wieland conjecture}
\label{subsection:prior-results}

One of the main results of \cite{landesmanL:canonical-representations} implies
that $H^1(\Sigma_{g',n'}, \mathbb C)$ has no nonzero vectors with finite orbit
when $g^2 > \# H$. Here is the precise statement.

\begin{theorem}[\protect{\cite[Theorem 1.4.1]{landesmanL:canonical-representations}}]
	\label{theorem:asymptotic-putman-wieland-no-irreps}
	With notation as in \autoref{subsection:putman-wieland-statement},
	for any $H$ cover $\Sigma_{g',n'} \to \Sigma_{g,n}$,
	$H^1(\Sigma_{g',n'}, \mathbb C)$ has no finite orbits under the action
	of $\Gamma$ whenever $\# H < g^2$.
\end{theorem}
This follows from \autoref{corollary:asymptotic-putman-wieland} below, since any
irrep of a finite group $H$ has dimension at most $\sqrt{\#H}$. 

The proof of \autoref{corollary:asymptotic-putman-wieland} given in \cite{landesmanL:canonical-representations}
somewhat obscures what is going on for two reasons. First, there we opted to give an indirect proof which was shorter given our other results
in that paper. Second, the proof was complicated due to the technical necessity of dealing with marked
points. 
In \autoref{section:proof-idea}, we give a more streamlined account of the idea, which we hope will be easier to digest.
We also try to motivate the vector bundle methods described there with an
alternative approach to proving
\autoref{theorem:asymptotic-putman-wieland-no-irreps} via elementary projective geometry in \autoref{section:geometric-approach}. 

\begin{remark}
	\label{remark:}
	While \autoref{theorem:asymptotic-putman-wieland-no-irreps} shows 
$H^1(\Sigma_{g',n'}, \mathbb C)$ contains no non-zero vectors with finite orbit  when $\#H < g^2$,
the same is not true of $H_1(\Sigma_{g',n'}, \mathbb C)$, which is dual to 
$H^1(\Sigma_{g',n'}, \mathbb C)$ by the universal coefficient theorem.
The homology classes corresponding to loops around punctures have
finite orbit under the mapping class group.
This tells us that when $n' > 1$, $H^1(\Sigma_{g',n'}, \mathbb C)$ has an $n'-1$ dimensional space
of coinvariants under the action of a finite index subgroup of the mapping class group, but
it has trivial invariants under the action of this subgroup. In particular, the $\Gamma$-action on $H^1(\Sigma_{g', n'}, \mathbb{C})$ is not semisimple. We have opted to state most of our results in terms of homology rather than cohomology as the statements are a bit cleaner in this formulation. 
\end{remark}

\begin{remark}
When the hypotheses of \autoref{theorem:asymptotic-putman-wieland-no-irreps} are satisfied, it implies the corresponding case of \autoref{conjecture:putman-wieland-intro}. Indeed, the action of $\Gamma$ on $H^1(\Sigma_{g', n'}, \mathbb{C})$ has as a natural subrepresentation its action on $H^1(\Sigma_{g'}, \mathbb{C})$, which hence has no fixed vectors. This action is semisimple, hence the same is true for the dual representation on $H_1(\Sigma_{g'}, \mathbb{C})$. 
\end{remark}

\subsection{Motivation for the Putman-Wieland conjecture: Ivanov's conjecture}
\label{subsection:ivanov}
Much of the original motivation for the Putman-Wieland conjecture was as an
approach to proving Ivanov's conjecture, which is a major open question in the
study of mapping class groups.
Ivanov's conjecture states that $\on{Mod}_{g,n}$ does
not virtually surject onto $\mathbb{Z}$ for $g > 2, n \geq 0$.
See \cite[\S 7]{Ivanov:problems} and \cite[Problem
2.11.A]{Kirby:problems}. 

As a historical note, it appears that Ivanov only posed this as a question, and
not a conjecture. However, since then, many sources have referred to this
question as Ivanov's
conjecture.

If we use $\on{I}_{g,n}$ to denote the statement that Ivanov's conjecture holds
for $(g,n)$ and $\on{PW}_{g,n}$ to denote the statement that the Putman-Wieland
conjecture holds for $(g,n)$, then the relation between Ivanov's conjecture and the
Putman-Wieland conjecture is the following:
\begin{theorem}[\protect{\cite[Theorem 1.3]{putmanW:abelian-quotients}}]
\label{theorem:ivanov-and-pw}
For $g \geq 3, n \geq 0$, 
$\on{PW}_{g-1, n+1} \implies \on{I}_{g,n}$ and $\on{I}_{g,n+1} \implies
\on{PW}_{g,n}$.
\end{theorem}
Technically, \cite[Theorem 1.3]{putmanW:abelian-quotients} states a slightly
different version of \autoref{theorem:ivanov-and-pw}.
However, the above formulation is equivalent, see \autoref{remark:pw-equivalent}.
Therefore, a proof of the Putman-Wieland conjecture would also yield a proof of
Ivanov's conjecture in almost all cases, and vice-versa.

\subsection{Motivation for the Putman-Wieland conjecture: big monodromy}

Another source of motivation for the Putman-Wieland conjecture comes from a
perspective on monodromy groups.
Indeed, a typical slogan in much of algebraic geometry and number theory is:
\begin{slogan}\label{slogan:big-monodromy}
\emph{Monodromy groups should be as big as possible.}
\end{slogan}
The Putman-Wieland conjecture \cite{putmanW:abelian-quotients} is implied by this philosophy.
Indeed, if one strongly believes \autoref{slogan:big-monodromy}, 
one might imagine that the intersection form and $H$-action on the cover $\Sigma_{g'}$ of
\autoref{conjecture:putman-wieland-intro} are the only constraints on the (virtual) action of $\on{Mod}_{g, 1}$, and so
one might guess the image of the action of a finite index subgroup of
$\on{Mod}_{g,1}$ on
$H_1(\Sigma_{g'}, \mathbb C)$ is via a finite-index subgroup of the centralizer of $H$ in $\on{Sp}_{2g'}(\mathbb{Z})$. 
The Putman-Wieland conjecture merely predicts this action has no nonzero vectors
with finite orbit. Even though the Putman-Wieland conjecture is substantially
weaker than determining the monodromy, there are no pairs $(g,n)$ with $g \geq
2$ for which it is known to hold.

\subsection{Counterexamples in low genus}

For $g\leq 1$, it turns out that the intersection form and the $H$-action are \emph{not} the only constraints on the (virtual) action of $\on{Mod}_{g,n+1}$ on $H_1(\Sigma_{g',n'}, \mathbb C)$; there are additional constraints arising from Hodge theory.  That is, \autoref{slogan:big-monodromy} is \emph{false}, if
naively interpreted. In \autoref{section:examples} we give a number of new examples where the action of $\on{Mod}_{g,n+1}$ factors through a smaller group than might be expected, for Hodge-theoretic reasons. 

For example, we give a
number of new counterexamples to the Putman-Wieland conjecture with $g
= 1, n  =1$, arising from \emph{origami curves}. 
These lead to interesting families of curves with high-dimensional
fixed parts in their Jacobians in
Examples \ref{example:quaternions}-\ref{example:4-dim-symplectic}.
Most of our counterexamples are produced via a Hodge-theoretic criterion for the failure of Putman-Wieland, 
\autoref{proposition:representation-theoretic-criterion}. These examples are the primary novelty of this paper.
In particular, we give a sequence of families of ``origami'' curves
whose Jacobians have isotrivial isogeny factors of arbitrarily large dimension in 
\autoref{example:2-dim-symplectic}, and include a number of open questions in
\autoref{subsection:questions}.
In \autoref{section:counterexample-hyperelliptic}, we also discuss recent counterexamples of
Markovi\'c for $g=2$, which admit a routine generalization to the hyperelliptic setting in any genus. In \autoref{rmk:sharpness}, we indicate how this hyperelliptic generalization demonstrates the sharpness of our methods for proving \autoref{theorem:asymptotic-putman-wieland-no-irreps}.

\subsection{Overview}

In \autoref{section:putman-wieland} we recall our main results toward the Putman-Wieland
conjecture.
In \autoref{section:proof-idea}, we sketch a streamlined version of our original
proof, while in
\autoref{section:geometric-approach} we describe our original, more geometric
approach to proving our results toward the Putman-Wieland conjecture.
We discuss counterexamples to the Putman-Wieland conjecture in low genus in
\autoref{section:examples}, and pose related questions in
\autoref{subsection:questions}.
We conclude with 
\autoref{section:counterexample-hyperelliptic}, giving an
exposition of Markovi\'c's genus $2$ counterexample to the Putman-Wieland
conjecture, based heavily on a theorem of Bogomolov-Tschinkel. We also explain
a sense in which this generalizes to higher genus hyperelliptic curves.

\subsection{Notation}
Throughout, we work over the complex numbers, unless otherwise stated.
For a pointed finite-type scheme or Deligne-Mumford stack $(X, x)$ over
$\mathbb{C}$, we will use $\pi_1(X, x)$ to denote the topological fundamental group 
 of the associated complex-analytic space or analytic stack.
We use Grothendieck conventions on projective space so that points of $\bp V$
correspond to hyperplanes in $V$, as opposed to lines in $V$.

\subsection{Acknowledgements}
We would like to thank
Paul Apisa,
Marco Boggi,
Simion Filip,
Joe Harris, 
Eduard Looijenga,
Eric Larson,
Dan Margalit,
Vladimir Markovi\'{c},
Carlos Matheus,
Curtis McMullen,
Andy Putman,
Will Sawin,
Ravi Vakil,
Isabel Vogt,
Melanie Matchett Wood
and Alex Wright
for helpful discussions related to this paper.
We also thank an anonymous referee for many extremely helpful comments.

This version of the article has been accepted for publication, after peer review
but is not the Version of Record and does not reflect post-acceptance improvements, or any
corrections. The Version of Record is available online at:
\url{http://dx.doi.org/10.1007/s40879-023-00637-w}

\section{The Putman-Wieland conjecture}
\label{section:putman-wieland}

In this section we recall some notation related to the Putman-Wieland conjecture, as well as the main results toward it obtained in \cite{landesmanL:canonical-representations}. 
That paper proves some substantially more general results; we explain them here in the special case concerning the Putman-Wieland conjecture.
Although we already stated the Putman-Wieland conjecture in \autoref{conjecture:putman-wieland-intro}, it will be useful to have a slightly different version of it that also incorporates
the covering group $H$. We next introduce notation to state this version
precisely.

\begin{notation}
	\label{notation:prym}
Let $\Sigma_g$ be an oriented surface of genus $g \geq 0$, and let
$\Sigma_{g,n}$ be the complement of $n\geq 0$ disjoint points in $\Sigma_g$, so
that $\Sigma_{g,n}$ is an oriented genus $g$ surface with $n$ punctures. We let $\on{PMod}_{g,n}$ be the subgroup of the mapping class group of $\Sigma_{g,n}$ fixing the punctures.
Fix a basepoint $v_0 \in \Sigma_{g,n}$, which we count as an additional puncture to
obtain an action of $\on{PMod}_{g,n+1}$ on $\pi_1(\Sigma_{g,n},v_0)$.
Suppose we are given a finite covering
$h: \Sigma_{g',n'} \to \Sigma_{g,n}$.
Let $\Gamma \subset \on{PMod}_{g,n+1}$ denote the finite index subgroup preserving
the covering $h$.

We obtain an action of $\Gamma$ on $H_1(\Sigma_{g',n'},\mathbb C)$.
By filling in the punctures, we also obtain
an action of $\Gamma$ on $V_h := H_1(\Sigma_{g'}, \mathbb C)$, referred to in 
\cite[p. 80-81]{putmanW:abelian-quotients} as a {\em higher Prym
representation}.

In the case $h$ is Galois with covering group $H$,
we may view
$H_1(\Sigma_{g',n'},\mathbb C)$ as an $H$-representation.
If $\rho$ is an irreducible $H$-representation we let 
\begin{align*}
	H_1(\Sigma_{g',n'},\mathbb C)^\rho := \rho \otimes \hom_H(\rho,
H_1(\Sigma_{g',n'},\mathbb C))
\end{align*}
denote the $\rho$-isotypic component.
In this case, we obtain an action of $\Gamma$ on the characteristic subrepresentation
$H^1(\Sigma_{g',n'},\mathbb C)^\rho \subset H^1(\Sigma_{g',n'},\mathbb C)$.

\end{notation}

\begin{definition}
	\label{definition:pw}
	Fix 
	non-negative integers
	$g$ and $n$
	and a finite covering 
	$h: \Sigma_{g',n'} \to \Sigma_{g,n}$
	Let $\PW_{g,n}^H$ be the statement that for any Galois
	$H$-covering $h$
	and $v \in V_h$ any nonzero vector,
	$v$ has infinite orbit under $\Gamma$.
\end{definition}

\begin{conjecture}[Putman-Wieland, ~\protect{\cite[Conjecture
	1.2]{putmanW:abelian-quotients}}]
	\label{conjecture:putman-wieland}
	$\PW_{g,n}^H$ holds for every group $H$ with $g \geq 2, n \geq 0$.
\end{conjecture}
\begin{remark}
It is known that the Putman-Wieland conjecture is not in general true for $g=2$; see \cite[Theorem 1.3]{markovic} or \autoref{proposition:hyperelliptic-counterexample} below.
\end{remark}

\begin{remark}
\label{remark:pw-equivalent}
We note that there are several differences between the original
\cite[Conjecture 1.2]{putmanW:abelian-quotients}
and \autoref{conjecture:putman-wieland}, but the statements are equivalent, see
also \cite[Remark 6.1.4]{landesmanL:canonical-representations}.
First, 
\cite[Conjecture 1.2]{putmanW:abelian-quotients}
is stated without a group $H$ and with $\mathbb Q$ coefficients instead of
$\mathbb C$ coefficients.
Additionally, 
\cite[Conjecture 1.2]{putmanW:abelian-quotients}
keeps track of boundary components in addition to punctures, but it is enough to
treat the case of only punctures by \cite[Lemma
6.2.5]{landesmanL:canonical-representations}.
\end{remark}

We now recall the main result of \cite{landesmanL:canonical-representations}
toward the Putman-Wieland conjecture.

\begin{theorem}[\protect{\cite[Theorem 6.2.1]{landesmanL:canonical-representations}}]
	\label{theorem:asymptotic-putman-wieland}
	With notation as in \autoref{notation:prym}, let $\rho$ be an irreducible
	complex $H$-representation and let $\Gamma' \subset \Gamma$ be a finite-index subgroup. Then
	$H^1(\Sigma_{g',n'},\mathbb C)^\rho$
	has no non-zero $\Gamma'$-invariant subrepresentations 
	of dimension strictly less than $2g - 2\dim \rho$.
	The same holds for
	$H_1(\Sigma_{g'},\mathbb C)^\rho$ in place of 
	$H^1(\Sigma_{g',n'},\mathbb C)^\rho$.
\end{theorem}

Because the span of a fixed vector is a $1$-dimensional subrepresentation, and $1
< 2g -2 \dim \rho$ if and only if $\dim \rho < g$, 
we get the following important corollary:

\begin{corollary}[\protect{\cite[Corollary 6.2.3]{landesmanL:canonical-representations}}]
	\label{corollary:asymptotic-putman-wieland}
	$\PW_{g,n}^H$ holds for every group $H$ such that every irreducible
	representation $\rho$ of $H$ has dimension $\dim \rho < g$.
\end{corollary}
In particular, the hypotheses of \autoref{corollary:asymptotic-putman-wieland}
are satisfied if $\#H<g^2$, implying
\autoref{theorem:asymptotic-putman-wieland-no-irreps}.

\section{The algebraic proof of
\autoref{corollary:asymptotic-putman-wieland}}
\label{section:proof-idea}

In this section, we outline the idea of the proof of \autoref{corollary:asymptotic-putman-wieland}, stripping away some of the additional technical complications from 
\cite{landesmanL:canonical-representations}.

We restrict to the unramified case, as the ramified case is analogous,
once one replaces usual vector bundles with parabolic vector bundles.

To start, given $\phi:\pi_1(\Sigma_g)\twoheadrightarrow H$, we set up a parameter space of $H$-covers. We begin by finding a diagram
\begin{align}\label{families of curves}\xymatrix{
\mathscr{X} \ar[r]^f \ar[rd]_{\pi'} & \mathscr{C} \ar[d]^\pi \\
& \mathscr{M}
}\end{align}
where $\pi$ is a relative curve of genus $g$, $\pi'$ is a relative curve of
genus $g'$, and $f$ is an $H$-cover, such that the map $\mathscr{M}\to
\mathscr{M}_g$ induced by $\pi$ is dominant \'etale, and such that for $m\in
\mathscr{M}$, the fiber $f_m: X\to Y$ is a finite \'etale connected Galois
$H$-cover corresponding to $\phi$.

	If $\PW^H_{g,0}$ were false, there would be a nonzero vector $v
	\in H_1(X, \mathbb Q)$ fixed by a finite index-subgroup of
	$\on{Mod}_{g,0}=\pi_1(\mathscr{M}_g)$. After replacing $\mathscr{M}$ with a finite \'etale cover, we may thus assume $v$ is in fact fixed by $\pi_1(\mathscr{M})$.
	This implies that the natural map $\nabla: H^0(R^1 \pi'_* \mathbb C \otimes \mathscr
	O_{\mathscr M}) \to H^0(R^1 \pi'_* \mathbb C \otimes \Omega_{\mathscr M})$
	induced by the Gauss-Manin connection has non-zero kernel. Setting $\mathscr{E}=R^1\pi'_*\mathbb{C}\otimes \mathscr{O}_{\mathscr{M}}$,
	Hodge theory gives a filtration $$\pi'_* \Omega_{\mathscr X/\mathscr
	M} =: F^1 \mathscr E \subset \mathscr E := R^1 \pi_* \mathbb C \otimes \mathscr O_{\mathscr M}$$
	with quotient $\mathscr E/F^1 \mathscr E \simeq R^1 \pi'_* \mathscr
	O_{\mathscr X}$.
	The theorem of the fixed part tells us that the kernel of the
	Gauss-Manin connection  in fact arises from a sub-Hodge structure of $R^1\pi'_*\mathbb{C}$, and hence meets
	$F^1 \mathscr E$ nontrivially. This type of argument is spelled out in
	more detail in \autoref{lemma:trivial-equivalence}.
	The connection $\nabla$, which is only a $\mathbb C$-linear map,
	induces a $\mathscr O_{\mathscr M}$-linear map $$\overline{\nabla}: \pi'_* \Omega_{\mathscr
	X/\mathscr M} \to R^1 \pi'_* \mathscr O_{\mathscr X} \otimes
	\Omega_{\mathscr M},$$
	which again has a nontrivial kernel, using the theorem of the fixed
	part.

	We now relate the kernel of $\overline{\nabla}$ to the vanishing
	of a particular map on cohomology.
	As above, fix a point $m \in \mathscr M$, whose fiber yields a finite \'etale cover $f_m: X \to Y$.
	The map $\overline{\nabla}_m$ may be identified as a map 
	$$\overline{\nabla}_m: H^0(X, \omega_{X}) \to H^1(X, \mathscr O_{X}) \otimes H^0(Y,
		\omega_Y^{\otimes 2}).$$
		Using Serre duality and the identification $H^0(X, \omega_{X})
		\simeq H^0(Y, f_{m*} \omega_{X})$, this can be identified with
	the map
\begin{align}\label{align:theta-map}
	\theta: H^0(Y, f_{m*} \omega_{X}) &\to \hom(H^0(Y, f_{m*} \omega_{X}),
	H^0(Y,\omega_Y^{\otimes 2})) \\
	s &\mapsto (t \mapsto \on{tr}(s \otimes t)).\nonumber
\end{align}
	By assumption $\theta$ has a non-zero kernel. Any non-zero element of the kernel  yields a nonzero map
	$f_{m*} \omega_{X} \to \omega_Y^{\otimes 2}$, inducing the zero map on
	global sections.

	Finally, we use the vanishing of the above map on cohomology to produce
	a stable vector bundle of low rank and high slope on $Y$ which is not generically
	globally generated, furnishing a contradiction. Given a representation
	$\rho: \pi_1(Y)\to \on{GL}_r(\mathbb{C})$, we let $\underline{\rho}$ be the associated local system on $Y$ and set $E^\rho=\underline\rho\otimes \mathscr{O}_Y$. Note that if $\rho$ has finite (or, indeed, unitary) image, then $E^\rho$ is semistable of slope zero.
	We can identify $f_{m*} \mathscr O_{X} \simeq \oplus_{\text{$H$-irreps
	$\rho$}} (E^\rho)^{\oplus \dim \rho}$, by viewing $f_{m*} \mathbb C$ as a local
	system on $Y$ induced by the the regular representation of $H$, and
	tensoring up with $\mathscr O_{Y}$.
Hence by the projection formula,
$$f_{m*}\omega_{X}=f_{m*}(f_m^*\omega_Y)=\omega_Y\otimes f_{m*}\mathscr{O}_{X}= \omega_Y\otimes \left(\bigoplus_{\text{$H$-irreps
	$\rho$}} (E^\rho)^{\oplus \dim \rho}\right).$$
	As we have produced a nonzero map
	$f_{m*} \omega_{X} \to \omega_Y^{\otimes 2}$, inducing the zero map on
	global sections, there must be some $\rho$ and a nonzero
	map $E^\rho \otimes \omega_Y \to \omega_Y^{\otimes 2}$ 
	which induces the $0$ map on global sections.
	However, $E^\rho \otimes \omega_Y$ is then a semistable vector bundle of slope
	$2g-2$, all of whose global sections factor through the subbundle $U := \ker (E^\rho \otimes
	\omega_Y \to \omega_Y^{\otimes 2})$.
	This will force $U$ to have a very high-dimensional space of global
	sections relative to its rank. 

	We now explain why the condition that $H^0(Y, U) = H^0(Y, E^\rho \otimes
	\omega_Y)$ leads to a contradiction.
	This is in tension with a version of Clifford's theorem for vector
	bundles, which says that if $U$ is a vector bundle which has a
	filtration by semistable bundles of slopes $\geq 0$ and $\leq 2g$, 
	then $h^0(Y, U) \leq \frac{\deg U}{2} + \rk
	U$
	\cite[Lemma 6.2.1]{landesmanL:geometric-local-systems}.
	This idea is used to prove (without too much difficulty) the more precise
	\cite[Proposition 6.3.1]{landesmanL:geometric-local-systems}
	which tells us that such a $U$ can only exist when
	$\dim \rho = \rk E^\rho \geq g$.
	We now spell this proof out in the case $\rk U = \rk E^\rho\otimes \omega_Y
	-1$.
	Note first that $\mu(U) \leq \mu(E^\rho\otimes \omega_Y
	)$ by semistability of $E^\rho\otimes \omega_Y$ and so $\mu(U) \leq
	(2g-2) \rk U = (2g-2)(\rk E^\rho\otimes \omega_Y - 1)$.
	Riemann-Roch, Clifford's theorem, and the above observation together yield 	
	\begin{align*}
		\deg E^\rho \otimes \omega_Y + (1-g) \rk E^\rho \otimes \omega_Y
		&\leq h^0(Y, E^\rho \otimes \omega_Y)
		= h^0(U) \\
		& \leq \rk U + \deg
		U/2 \\
		&\leq (\rk E_\rho \otimes \omega_Y - 1) + (g-1)(\rk E_\rho \otimes
		\omega_Y - 1).
	\end{align*}
	Solving for $\rk E_\rho \otimes \omega_Y$ gives $\dim \rho =\rk E_\rho \otimes \omega_Y \geq g$.

	This verifies $\PW_{g,0}^H$, since we assumed $H$ was a group all of
	whose representations had rank less than $g$.
	\begin{remark}
		In fact the methods here can be used to prove a stronger statement, about families of curves that do not necessarily dominate $\mathscr{M}_g$. Indeed, let 
		$$\xymatrix{
\mathscr{X} \ar[r]^f \ar[rd]_{\pi'} & \mathscr{C} \ar[d]^\pi \\
& \mathscr{M}
}$$
be as in \eqref{families of curves}, except we now assume that the map
$\mathscr{M}\to \mathscr{M}_g$ has image of codimension $\delta$; we assume
$\mathscr{M}\to \mathscr{M}_g$ is \'etale onto its image. Arguing as above, we
find that $H^0(E^\rho\otimes \omega_Y)\to H^0(\omega_Y^{\otimes 2})$ has rank
$\delta$, and so replacing the use of \cite[Proposition
6.3.1]{landesmanL:geometric-local-systems} with \cite[Proposition
6.3.6]{landesmanL:geometric-local-systems}, we obtain:
\begin{theorem}	\label{thm:high-codim-version}
With notation as above, suppose $\dim\rho<g-\delta$. Then the action of
$\pi_1(\mathscr{M}, m)$ on $H^1(\mathscr{X}_m, \mathbb C)^\rho$ has no non-zero finite orbits.
\end{theorem}
Analogously to \autoref{theorem:asymptotic-putman-wieland-no-irreps}, one
immediately obtains that for $|H|<(g-\delta)^2$, the action of
$\pi_1(\mathscr{M}, m)$ on $H^1(\mathscr{X}_m, \mathbb C)$ has no non-zero vectors with finite orbit. Replacing usual vector bundles with parabolic vector bundles, one may prove an analogue of \autoref{thm:high-codim-version} for ramified covers.
	\end{remark}

\section{A geometric approach to \autoref{theorem:asymptotic-putman-wieland}}
\label{section:geometric-approach}

While we were thinking about the Putman-Wieland conjecture, we first came up
with a rather different argument for proving
\autoref{theorem:asymptotic-putman-wieland}, which involves analyzing the
geometry of the canonical map, and in particular the quadrics which contain the
covering curve under its canonical map.
We discovered that a certain degeneracy locus naturally
appeared in the intersection of these quadrics, which yields a different method to
produce the same non-generically globally
generated vector bundles that are described in \autoref{section:proof-idea}.
It was only after we found this geometric
approach that we were able to somewhat simplify the presentation via the
algebraic argument described in \autoref{section:proof-idea}.
In this section, we aim to present our original, geometric approach.
It is our hope that sharing this more geometric approach may shed some additional light
on the problem, and may also be of use in related problems.
Since we are only trying to convey the main idea, as in
\autoref{section:proof-idea},
we will assume $n = 0$, i.e., the relevant cover $X \to Y$ is \'etale.

The first steps in
this geometric approach are the same as those described in \autoref{section:proof-idea}.
Namely, by studying the Gauss-Manin connection, we find that given a counterexample to Putman-Wieland, the natural
map
$H^0(X, \omega_{X}) \xrightarrow{\theta} \hom(H^0(X, \omega_{X}), H^0(Y,
\omega_Y^{\otimes
2}))$
has non-zero kernel.
If $X \to Y$ is an $H$-cover, we can decompose the above as
$H$-representations, and there must be some irreducible $H$-representation
$\rho$ and a subrepresentation $\rho \subset H^0(X, \omega_{X})$ in the kernel
of $\theta$.
Let us now try to understand what this copy of $\rho$ buys us in terms of the
canonical map $X \to \mathbb P H^0(X, \omega_X)$.
Recall here we are using Grothendieck conventions, so points of $\mathbb P
H^0(X, \omega_X)$ correspond to hyperplanes in $H^0(X, \omega_X)$.
By inspection of \eqref{align:theta-map}, the map $$H^0(X, \omega_X)\otimes H^0(X, \omega_X)\to H^0(Y, \omega_Y^{\otimes 2})$$ adjoint to $\theta$ factors through $\on{Sym}^2 H^0(X, \omega_X)$.
Hence, restricting to $\rho\otimes H^0(X, \omega_X)$ we obtain that the composition 
\begin{equation}
\rho \otimes H^0(X, \omega_X) \xrightarrow{\alpha_1} \sym^2 H^0(X, \omega_X) \to H^0(Y,
\omega_Y^{\otimes 2}),
\label{equation:symmetric-map-repeated}
\end{equation}
vanishes.
Observe that $H^0(Y, \omega_Y^{\otimes 2})$ carries the trivial
$H$-representation, and hence $\rho$ can only pair nontrivially under the above composition with 
$H^0(X, \omega_X)^{\rho^\vee}$, the $\rho^\vee$ isotypic part of $H^0(X,
\omega_X)$.
Hence, the vanishing of \eqref{equation:symmetric-map-repeated} is equivalent to
the vanishing of
\begin{equation}
	\label{equation:symmetric-map-isotypic}
\begin{tikzpicture}[baseline= (a).base]
\node[scale=.7] (a) at (0,0){
				\begin{tikzcd}
		\rho \otimes H^0(X, \omega_X)^{\rho^\vee}  \ar
		{r}{\alpha_1} \ar[rr,bend left=10, ,"\alpha"]&  \sym^2 H^0(X,
	\omega_X)  \ar {r}{\alpha_2} & (\sym^2 H^0(X, \omega_X))^H\to H^0(X, \omega_X^{\otimes 2})^H =H^0(Y,
	\omega_Y^{\otimes 2}),
				 \end{tikzcd}   
};
\end{tikzpicture}
\end{equation}
where $\alpha_1$ is the multiplication map, $\alpha_2$ is the averaging map $$q
\mapsto \frac{1}{\# H}\sum_{h \in H} h^* q,$$ and
$\alpha := \alpha_2 \circ \alpha_1$.
Indeed, viewing $\sym^2 H^0(X, \omega_X)$ as the space of quadrics on $\mathbb{P}H^0(X, \omega_X)$, any element of the image of $\alpha$ yields a quadric vanishing on the image of $X$ under the canonical map.

Let us derive some consequences from this observation. To get a feeling for what is going on, we work out examples depending
on the rank of $\rho$, so as to motivate the general case.
Suppose $Y$ has genus $g$.
By the Chevalley-Weil formula,
(see the original source \cite{chevalleyW:uber} or the more expository \cite[Theorem 2.1]{naeff:chevalley},)
$H^0(X, \omega_X)^{\rho^\vee}$ is a direct sum of $(g - 1) \cdot \dim \rho^\vee$
many copies of $\rho^\vee$, and so has dimension $(g - 1) \cdot (\dim
\rho)^2$.

\subsection{Summary of the remainder of this section}
\label{subsection:summary}
We first see how to use the above observation to rule out $1$-dimensional representations in
\autoref{subsection:dim-1}.
We use $r := \dim \rho$.
Following this, we will see how to use the above to produce a vector bundle
$E^\rho \otimes \omega_Y$ which is not generically globally generated in the
cases $(r =2, g=2)$ in \autoref{subsection:dim-2}, $(r = 3, g =2)$ in
\autoref{subsection:dim-3}, and then the general case in
\autoref{subsection:general-case}.
We stress that in these examples, we are not trying to prove or disprove
the Putman-Wieland conjecture, but merely to indicate a geometric reason for why $E^\rho \otimes
\omega_Y$ is not generically globally generated.
One can then argue as in \autoref{section:proof-idea} that any vector bundle
which is not generically globally generated must have rank $E^\rho \otimes
\omega_Y \geq g$ using
Clifford's Theorem, and so
this would lead to a proof of the Putman-Wieland conjecture in the case
$\dim \rho < g$.
Again, there is a valid alternate argument given in
\autoref{section:proof-idea} of why $E^\rho \otimes \omega_Y$ fails to be
generically globally generated. The point of this section is to give an
alternate, more geometric explanation of that fact.

\subsection{The dimension of $\rho$ is $1$}
\label{subsection:dim-1}

We first indicate how to rule out the case that $\rho$ has rank $1$.
The main idea is that the vanishing of \eqref{equation:symmetric-map-isotypic}
yields a quadric of rank $2$ containing the image of $X$ under the canonical
map. Since rank $2$ quadrics vanish on a union of two hyperplanes, this would
force the image of $X$ under the canonical map
to be contained in a hyperplane, contradicting nondegeneracy of the canonical
map.
This conclusion is established in \autoref{remark:no-1-dim}

\begin{lemma}
	\label{lemma:rank-2d-quadric}
	Let $\alpha$ be as in \eqref{equation:symmetric-map-isotypic}.
	Any quadric in $\on{im}(\alpha) \subset (\on{Sym}^2 H^0(X, \omega_X))^H$ has rank at most $2\dim\rho$.
\end{lemma}
\begin{proof}
Because the multiplication map is $H$-equivariant, any given map $\rho \otimes
H^0(X, \omega_X) \to H^0(Y, \omega_Y^{\otimes 2})$, where the target has the
trivial $H$-action, must factor through a map 
$$\rho\otimes H^0(X, \omega_X)^{\rho^\vee}\to \rho \otimes \rho^\vee \to H^0(Y, \omega_Y^{\otimes 2})$$
for $  H^0(X, \omega_X)^{\rho^\vee}\to \rho^\vee$ an irreducible quotient representation
isomorphic to the dual of $\rho$.
This in turn factors through $\sym^2(\rho \oplus \rho^\vee)$, and so the
resulting quadric has rank at most $2 \dim \rho = \dim \rho \oplus \rho^\vee$.

The reason for the claimed factorization is that the given map factors through the quotient $\rho\otimes (H^0(X, \omega_X)^{\rho^\vee}/K),$ where $K$ is the subspace of $H^0(X, \omega_X)^{\rho^\vee}$ pairing to zero with $\rho$; the quotient $H^0(X, \omega_X)^{\rho^\vee}/K$ is isomorphic as an $H$-representation to $\rho^\vee$.
\end{proof}

\begin{proposition}
	\label{proposition:1-dimensional-kernel}
	Suppose that the genus of $Y$ is at least $2$. Then there are no $\rho$
	of dimension $1$ in the kernel of $\overline{\nabla}_m$.
\end{proposition}
\begin{proof}
Any quadric of rank at most $2$ is supported on a union of hyperplanes.
Since the canonical map is always nondegenerate, its image cannot be contained in a
hyperplane. Since $X$ is smooth and
connected, it also cannot be contained in the vanishing locus of a quadric of rank at
most $2$. But \autoref{lemma:rank-2d-quadric} implies that if a $1$-dimensional representation was in the kernel of the $\overline{\nabla}_m$, the image of the canonical map would be contained in such a quadric.
\end{proof}
\begin{remark}
\label{remark:no-1-dim}
As in \autoref{section:proof-idea},
\autoref{proposition:1-dimensional-kernel} immediately implies if
$\Sigma_{g'}\to \Sigma_g$ is an $H$-cover, with $g\geq 2$, and $\rho: H\to
\mathbb{C}^\times$ is one-dimensional, no non-zero vector in $H^1(\Sigma_{g'}, \mathbb{C})^\rho$ has finite orbit under the (virtual) action of $\on{Mod}_g$.
One might say this verifies the ``Putman-Wieland conjecture for $1$-dimensional
representations.'' A related, but somewhat stronger statement is originally due to Looijenga
 \cite{looijenga:prym-representations}.
\end{remark}

\subsection{The dimension of $\rho$ is $2$ and $g = 2$}
\label{subsection:dim-2}
To simplify matters somewhat, we will restrict to the case $g = 2$, where all
the main ideas are already present and we further assume that $\rho$ is {\em not
self-dual}. Though see \autoref{remark:non-self-dual-case} for a remark on what
happens in the self-dual case.

In the case $\dim \rho = 2$, we have a $2 (g-1)= 2$ dimensional space 
\begin{align*}
	\hom_H(\rho^\vee, H^0(X,\omega_X))
\end{align*}
by Chevalley-Weil, so $H^0(X, \omega_X)^{\rho^\vee} \simeq (\rho^\vee)^{\oplus
2}$. We can consider the $2 + 2 \cdot 2$ subspace of $H^0(X,
\omega_X)$ given by $\rho \oplus H^0(X, \omega_X)^{\rho^\vee},$
where here $\rho$ lies in the kernel of $\theta$ from \autoref{align:theta-map}.
From this, we get $2$ quadrics $Q_1$ and $Q_2$ of rank $4$ vanishing on the
image of $X$ under the projection of the canonical map to 
$\mathbb P(\rho \oplus H^0(X, \omega_X)^{\rho^\vee})$.
Here, $Q_1$ and $Q_2$
are nonzero quadrics, corresponding to the universal quadrics
expressing the incidence relation between
$\rho$ and the $2$ copies of $\rho^\vee$.
\begin{remark}
	\label{remark:non-self-dual-case}
	This crucially uses that $\rho$ is not self-dual, as in the case $\rho$ is
symplectically self-dual, one of these quadrics may be identically the $0$
quadric. 
What changes in the above is that 
$\rho$ is already contained in
$H^0(X, \omega_X)^{\rho^\vee},$
so $\rho \oplus H^0(X, \omega_X)^{\rho^\vee}$ is not a subspace of
$H^0(X,\omega_X)$.

To see that one of the quadrics will be $0$ in the symplectically self-dual case,
we use the fact that there there is an isomorphism $\rho \simeq
\rho^\vee$ inducing an injection $\wedge^2 \rho \to \rho \otimes \rho^\vee\simeq \rho\otimes \rho,$
(using that we are in characteristic $0$ to split the natural surjection) whose
projection to $\sym^2 \rho$ vanishes. This induces the $0$ quadric.
When $\rho$ is not self-dual, we always obtain two quadrics of rank $4$. When
$\rho$ is symplectically self-dual, we obtain one quadric of rank $4$, while
when $\rho$ is orthogonally self-dual, we obtain one quadric of rank $4$ and one
quadric of rank $3$.
\end{remark}

\begin{figure}
\includegraphics[trim={0 12cm 0 0},scale=.5]{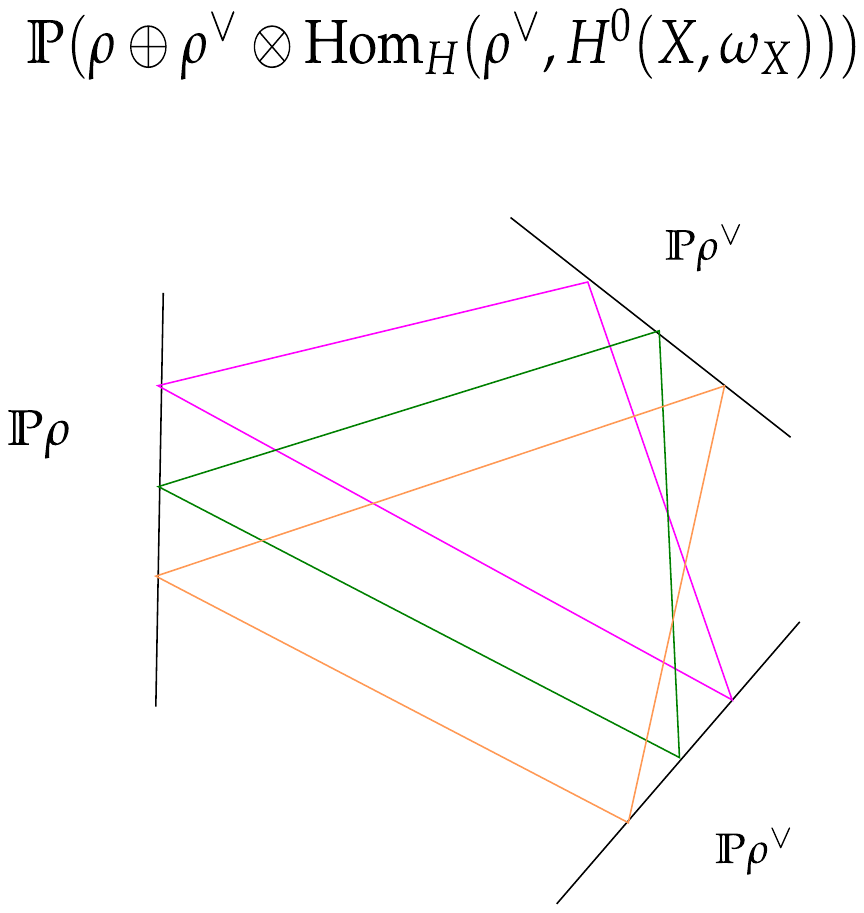}
\caption{A picture of the variety $V$ swept out by colored planes, which is a
component of the variety $Q_1 \cap Q_2$ cut out by
certain quadrics containing the canonical image of $X$ in the case that a $2$-dimensional
representation lies in the kernel of $\theta$.
Each colored plane corresponds to a point $x \in \bp \rho$ and is the span $x$
together with the corresponding
hyperplanes in each copy of $\bp \rho^\vee$ on which $x$ vanishes.
}
	\label{figure:2-dim}
\end{figure}

To make the situation even more concrete, continue supposing $\dim \rho = 2$ and $g =
2$, give $\rho$ coordinates $x,y$, and give the two copies of $\rho^\vee$
dual coordinates $u,v$ and $s,t$. Then, the two quadrics $Q_1$ and $Q_2$ can be concretely expressed as
$xu - yv = 0$ and $xs - yt = 0$.
The variety $Q_1 \cap Q_2$ is a degree $4$ variety which has two components:
it is the union of the codimension $2$
plane $P$ given by $x = y = 0 $ and the degree $3$ variety $V$ given by $xu - yv
= xs - yt = vs-tu=0$.
Since the curve $X$ is irreducible and nondegenerate, it cannot lie in the
component $P$
and so must be contained in $V$.
We have now managed to produce a new equation containing $X$, $vs - tu$,
which is only a function of the variables in $H^0(X, \omega_X)^{\rho^\vee}$ and
does not depend on the $\rho$ component. Moreover, it is the degeneracy locus of
the matrix
\begin{align*}
\begin{pmatrix}
s & t  \\
u & v
\end{pmatrix},
\end{align*}
which already gives a hint as to where the failure of generic global generation
of $E^\rho \otimes \omega$
may come from.

There is yet another geometric perspective on the variety 
$xu - yv = xs - yt = vs-tu=0$.
Starting with a point $[x,y]$ in $\mathbb P \rho$, we get corresponding hyperplanes
in $\mathbb P \rho^\vee$ at which $[x,y]$ vanishes. In the case $\dim \rho = 2$, this
hyperplane has codimension $1$ in the $1$-dimensional $\mathbb P \rho^\vee$, and hence
corresponds to a point.
The variety 
$xu - yv = xs - yt = vs-tu=0$ expresses the incidence relation between these two;
it parameterizes tuples $(x,y,s,t,u,v)$ so that $[x,y]$ vanishes on both $[s,t]$
and $[u,v]$. We will next see a similar phenomenon when $\dim \rho = 3$.

\subsection{The dimension of $\rho$ is $3$ and $g = 2$}
\label{subsection:dim-3}

We next consider the case $\dim \rho = 3$ and $g = 2$.
In this subsection, we will explain how to show $E^\rho \otimes \omega_Y$ fails to be generically
globally generated, which will not in general lead to a contradiction, but which
motivates our arguments in \autoref{section:proof-idea}, see
\autoref{subsection:summary}.
As in the dimension $2$ case of \autoref{subsection:dim-2}, we obtain a variety cut out by $3$ quadrics
expressing the incidence between $\rho$ and $H^0(X, \omega_X)^{\rho^\vee} \simeq
(\rho^\vee)^{\oplus 3}$.
Viewing this as a subvariety of $\mathbb P (\rho \oplus (\rho^\vee)^{\oplus 3})$
we find this degree $8$, codimension $3$ subvariety is the union of the codimension $3$
plane $\rho = 0$ and a certain degree $7$ codimension $3$ variety, which can be
expressed as the closure $V$ of the incidence variety 
\begin{align*}
\{ (z,a,b,c) \in \rho \oplus
(\rho^{\vee})^{\oplus 3} : z \neq 0, z(a) = z(b) = z(c) = 0\}.
\end{align*}
As in the previous case, the image of $X$ is irreducible and nondegenerate, so cannot lie on the codimension
$3$ plane $\rho = 0$, and hence lies in $V$.

\begin{figure}
\includegraphics[trim={0 12cm 0 0},scale=.5]{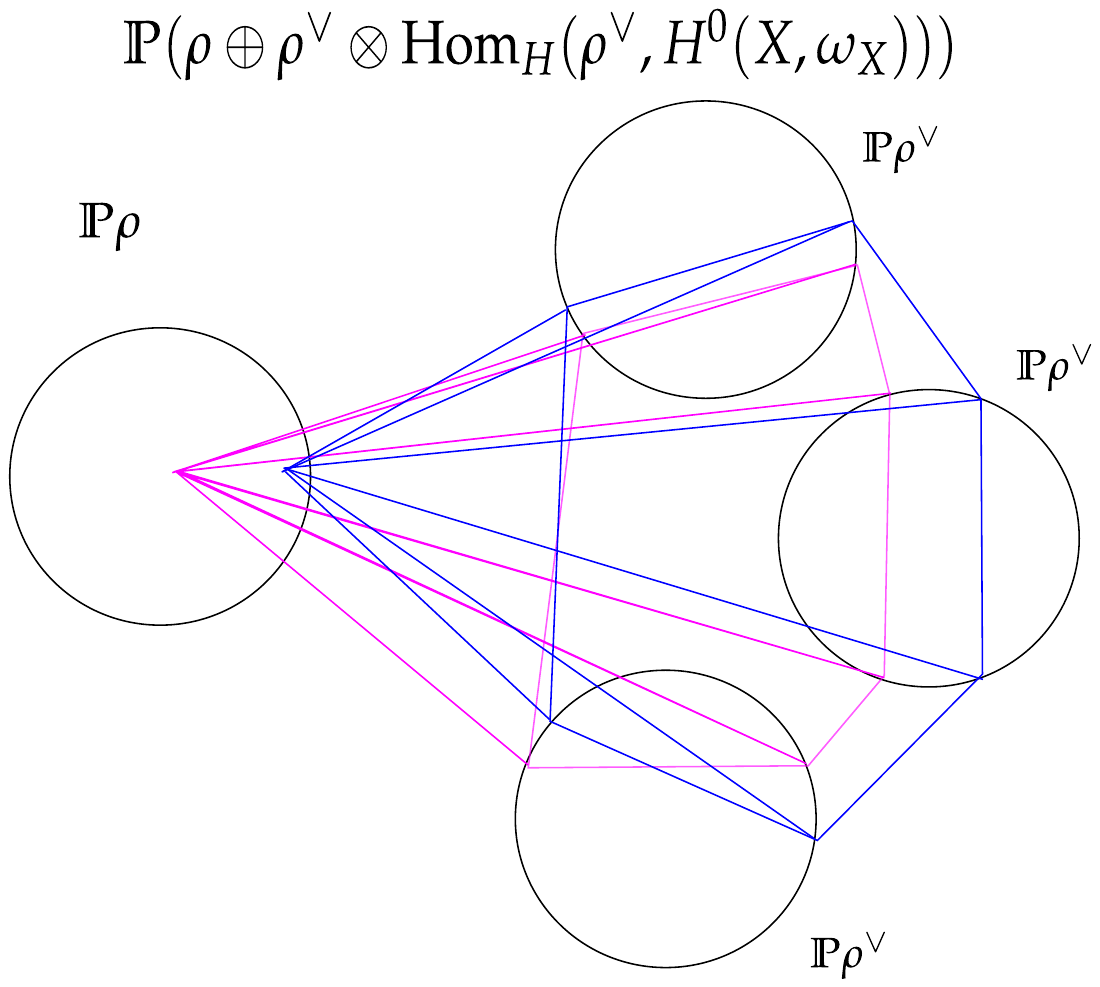}
\caption{A picture of the variety $V$ swept out by colored planes, which is
contained in the intersection of
certain quadrics containing the canonical image of $X$ when a $3$-dimensional
representation $\rho$ lies in the kernel of $\theta$.
Each colored plane corresponds to a point $x \in \bp \rho$ and is the span $x$
together with the corresponding
hyperplanes in each copy of $\bp \rho^\vee$ on which $x$ vanishes.
}
	\label{figure:3-dim}
\end{figure}

Next, let us try to understand how the non-generically globally generated
vector bundle appearing in \autoref{section:proof-idea}
is related to this geometric setup. 
To find it, we project $V$ away from $\mathbb P \rho
\subset \mathbb P (\rho \oplus (\rho^\vee)^{\oplus 3})$.
Under this projection, we get that $V$ projects to 
\begin{align*}
W := \{(a,b,c) \in (\rho^{\vee})^{\oplus 3} : \text{ there exists } z \neq 0, z(a) = z(b) = z(c) = 0\}.
\end{align*}

This can equivalently be expressed as
\begin{align*}
W = \{(a,b,c) \in (\rho^{\vee})^{\oplus 3} : (a,b,c) \text{ lie on a common line
in } \mathbb P \rho\}.
\end{align*}
The map $\iota: X \to \mathbb P (\rho \oplus \hom_H(\rho^\vee, H^0(X, \omega_X)))$ is induced by a sub-linear system of
$\omega_X$. That is, $\mathscr O_{\mathbb P (\rho \oplus \hom_H(\rho^\vee, H^0(X,
\omega_X)))}(1)$ pulls back to $\omega_X$, or a subsystem
thereof if $\omega_X$ has basepoints in $\rho \oplus \hom_H(\rho^\vee, H^0(X,
\omega_X)) \subset H^0(\omega_X)$.
Therefore, the above map $\iota$
is induced by a map of sheaves
$H^0(X, \omega_X)^{\rho^\vee} \otimes \mathscr O_X \to \omega_X$.
Since $\rho^\vee \otimes \hom_H(\rho^\vee, H^0(X, \omega_X)) \simeq H^0(X,
\omega_X)^{\rho^\vee}$, this also corresponds to a map
$\rho^\vee \otimes \hom_H(\rho^\vee, H^0(X, \omega_X)) \to \omega_X$ or
equivalently
$\psi: \hom_H(\rho^\vee, H^0(X, \omega_X))\otimes \mathscr O_X \to \rho \otimes \omega_X$.
This is an $H$-equivariant map, and so descends to a map of vector bundles on
$Y$.

Under the above translation, the condition defining $W$, that $a,b,c$ are 
collinear,
corresponds to the map $\psi$ not having maximal rank. 
Here $\rho$ denotes a trivial bundle of rank $r$ on $X$ with a specified
$H$-action, and we use $E^\rho$ to denote the corresponding descent to $Y$,
which pulls back to $\rho$ on $X$.
Note that we can identify $H$-equivariant maps 
\begin{align*}
\hom_H(\rho^\vee, H^0(X, \omega_X)) \simeq
\hom_H(\rho^\vee, \hom_X(\mathscr O_X, \omega_X)) \simeq \hom_H(\mathscr O_X, \rho
\otimes \omega_X).
\end{align*}
Upon descending to $Y$, elements of this vector space correspond to elements of
$H^0(Y, E^\rho \otimes \omega_Y)$.
From this, the map $\psi$ descends to a map of vector bundles on $Y$,
\begin{align*}
\xi: H^0(Y, E^\rho \otimes \omega_Y) \otimes \mathscr O_Y \to E^\rho \otimes
\omega_Y.
\end{align*}
Since the pullback $\xi|_X$ factors through a proper sub-bundle of $E^\rho \otimes
\omega_Y$, as was shown above, the same holds true of $\xi$.
Further, unwinding the definition of $\xi$ shows it is none other than the
natural evaluation map on global sections. This implies that $E^\rho \otimes
\omega_Y$ is not generically globally generated.

We note that at this point we have not ruled out the possibility that
Putman-Wieland fails in this $r = 3, g =2$ case because $3 = r \geq
g = 2$, see \autoref{subsection:summary}. All we have shown that 
$E^\rho \otimes \omega_Y$ fails to be generically globally generated.

\subsection{The general case}
\label{subsection:general-case}

There are essentially no new ideas in generalizing the lack of global generation observed above beyond the $\dim \rho = 3, g = 2$
case.
We consider only the case $\rho$ is {\em not self-dual}, but the self-dual case can be
handled similarly, with appropriate modifications. 
See \autoref{remark:non-self-dual-case} for an example of this.
We will explain how to show $E^\rho \otimes \omega_Y$ fails to be generically
globally generated. By combining this with the argument using Clifford's Theorem
for vector bundles in \autoref{section:proof-idea}, we find that we must have
$\dim \rho \geq g$,
see also \autoref{subsection:summary}

Say $\dim \rho = r$ and $Y$ has genus $g$.
By Chevalley-Weil, we obtain a variety cut out by $r(g-1)$ quadrics
expressing the incidence between $\rho$ and $H^0(X, \omega_X)^{\rho^\vee} \simeq
(\rho^\vee)^{\oplus r(g-1)}$.
As in \autoref{subsection:dim-3}, these quadrics cut out a reducible variety
in $\mathbb P(\rho \oplus (\rho^\vee)^{\oplus r(g-1)})$,
one of whose components is the plane $\rho = 0$, with the other being the closure of
the subvariety $\{ (z,a_1, \ldots, a_{r(g-1)} \in \rho \oplus
(\rho^{\vee})^{\oplus r(g-1)} : z \neq 0, z(a_i) = 0\}$.
The canonical curve cannot lie in the plane, and thus lives in the latter
subvariety. Projecting this subvariety away from $\mathbb P \rho$ gives a map from $X$
to 
\begin{align*}
W := \{(a_1, \ldots, a_{r(g-1)}) \in (\rho^{\vee})^{\oplus r(g-1)} :
(a_1,\ldots,a_{r(g-1)}) 
\\
\text{ lie on a common hyperplane in } \mathbb P \rho \}.
\end{align*}
As in the three dimensional case of \autoref{subsection:dim-3}, the map $X \to \mathbb P H^0(X,
\omega_X)^{\rho^\vee}$ 
gives a map of vector bundles on $X$,
$\hom_H(\rho^\vee, H^0(X, \omega_X))\otimes \mathscr O_X \to \rho \otimes
\omega_X$.
The condition that $X$ factors through $W$ tells us this map of vector
bundles drops rank.
Letting $E^\rho$ denote the descent of $\rho$ to $Y$,
the above map descends to $Y$ and we obtain a map
of vector bundles
$\xi: H^0(Y, E^\rho \otimes \omega_Y) \otimes \mathscr O_Y \to E^\rho \otimes
\omega_Y$
which again drops rank and corresponds to the natural evaluation map.
Because this evaluation map drops rank, $E^\rho \otimes \omega_Y$ is not generically globally
generated. 
By combining this with the argument using Clifford's Theorem
for vector bundles in \autoref{section:proof-idea}, we find that we must have
$\dim \rho \geq g$, as described in \autoref{subsection:summary}.

\section{Examples}
\label{section:examples}
In this section, we give several counterexamples to the Putman-Wieland
conjecture in low genus.
Starting in \autoref{subsection:low-genus-counterexamples}, we give a method to produce genus $0$ and $1$ counterexamples to Putman-Wieland.
We give a number of examples, showing that this realizes all the genus $0$
cyclic counterexamples which McMullen found \cite{mcmullen:braid-groups}, as
well as the ``Eierlegende-Wollmilchsau" genus $1$ counterexample in the original
Putman-Wieland paper \cite[Appendix A]{putmanW:abelian-quotients}.
We also produce many new ``origami'' counterexamples in genus $1$. In particular,
\autoref{example:2-dim-symplectic} gives an infinite sequence of families of ``origami"
curves, whose Jacobians have isotrivial isogeny factors of arbitrarily large dimension.
We conclude the section with a number of questions about creating families of
curves whose Jacobians have isotrivial isogeny factors.

\subsection{Genus 0 and 1 counterexamples}
\label{subsection:low-genus-counterexamples}
There are many cases where $\PW_{0,n}$ and $\PW_{1,n}$ fail to hold (see e.g.
\cite[Theorem 8.3]{mcmullen:braid-groups} and ~\cite[Appendix A]{putmanW:abelian-quotients}). From our point of view, these examples admit purely Hodge- and representation-theoretic explanations.

Let us now formally define what it means for a cover to give a counterexample to
Putman-Wieland.
\begin{definition}
	\label{definition:pw-counterexample}
	Suppose $h: \Sigma_{g',n'} \to \Sigma_{g,n}$
is a finite covering of topological surfaces.
Let $\Gamma \subset \on{PMod}_{g,n+1}$
denote the finite index subgroup preserving $h$.
The action of $\Gamma$ on $\pi_1(\Sigma_{g',n'})$ induces an action on
$H_1(\Sigma_{g',n'}, \mathbb C)$ which preserves the subspace spanned by homology classes
of loops around punctures, and hence also induces an action on
$H_1(\Sigma_{g'}, \mathbb C)$.
We say $h$ {\em furnishes a counterexample to Putman-Wieland}, if
there is some nonzero $v \in H_1(\Sigma_{g'}, \mathbb C)$ with
finite orbit under $\Gamma$.

Let $Y$ be a compact Riemann surface of genus $g$ with $n$ marked points $p_1,
\ldots, p_n$.
Upon identifying $\pi_1(Y-\{p_1, \ldots, p_n\}) \simeq \pi_1(\Sigma_{g,n})$,
let $f: X \to Y$ be the
covering of compact Riemann surfaces, ramified at most over $p_1, \ldots, p_n$,
corresponding to the topological cover $h$.
If $h$ furnishes a counterexample to Putman-Wieland, we also say $f: X \to Y$ furnishes a counterexample to
Putman-Wieland.
\end{definition}
\begin{remark}
\label{remark:}
Even though the Putman-Wieland conjecture assumes $g \geq 2$, we still say $f: X
\to Y$ furnishes a counterexample to Putman-Wieland when $Y$ has genus $g \leq 1$.
\end{remark}

To state our criterion for producing counterexamples to Putman-Wieland, we set some notation to describe families of covers of curves.
\begin{notation}
	\label{notation:versal-family}
	We fix non-negative integers $(g,n)$ so that $n \geq 1$ if $g = 1$ and
	$n \geq 3$ if $g = 0$, i.e., $\Sigma_{g,n}$ is hyperbolic.
	Let $\mathscr{M}$ be a connected complex variety. A \emph{family
	of $n$-pointed curves of genus $g$ over $\mathscr{M}$} is a smooth
	proper morphism $\pi: \mathscr{C}\to \mathscr{M}$ of relative dimension one, with geometrically
	connected genus $g$ fibers, equipped with $n$ sections $s_1,\cdots,
	s_n: \mathscr{M}\to \mathscr{C}$ with disjoint images. Call such a family \emph{versal} if the
	induced map $\mathscr{M}\to \mathscr{M}_{g,n}$ is dominant and \'etale.
	Here $\mathscr M_{g,n}$ denotes the Deligne-Mumford moduli stack of $n$-pointed genus
	$g$ smooth proper curves with geometrically connected fibers.
\end{notation}

Recall that an irreducible representation $\rho$ is {\em symplectically self-dual} if
$(\wedge^2 \rho)^H \neq 0$, i.e., there is a nonzero map $\on{triv} \to \wedge^2
\rho$.
\begin{proposition}\label{proposition:representation-theoretic-criterion}
    With notation as in \autoref{notation:versal-family}, let $m\in \mathscr{M}$
    be a point. Let $X=\mathscr{X}_m, Y=\mathscr{C}_m$. Let $V=H^1(X,
    \mathbb{Q})/H^1(Y, \mathbb{Q})$. 
    Suppose that there exists an irreducible
    representation $$\rho: H\to \on{GL}_r(\mathbb{Q})$$ such that,
    letting $V^\rho$ denote the $\rho$-isotypic piece of $V$,
    $V^\rho\otimes
    \mathbb{C}\simeq \oplus_{i=1}^s \rho_i^{n_i}$ with $\rho_1, \ldots, \rho_s$
    irreducible and pairwise distinct. Suppose that either:
    \begin{enumerate}
        \item The Hodge decomposition of $V^\rho\otimes\mathbb{C}$ is an
	isotypic decomposition, i.e. there is a subset $S \subset \{1, \ldots,
	s\}$ so that $F^1V^\rho\otimes \mathbb{C} \simeq \bigoplus_{i \in S}
	\hspace{.1cm}
	\rho_i^{n_i}$, or
        \item for every $i$, $\rho_i$ is symplectically self-dual, and $\rho_i$
	appears with multiplicity at most $1$ in $F^1V^\rho\otimes \mathbb{C}$.
    \end{enumerate}
    Then $f: X \to Y$ furnishes a counterexample to Putman-Wieland.
\end{proposition}
Before proceeding with the proof, we comment on some restrictions on the
representations that
must be satisfied for the conditions to hold.
\begin{remark}
	\label{remark:}
	Note that in
	\autoref{proposition:representation-theoretic-criterion}(1),
	the representations $\rho_i$ appearing in $F^1 V^\rho$ must all be
	non-self dual. The reason for this is that, as $V^\rho \otimes \mathbb
	C$ carries a Hodge
	structure, the quotient $V^\rho \otimes \mathbb C/F^1 V^\rho \otimes
	\mathbb C$ is dual to $F^1 V^\rho \otimes \mathbb C$.
	However, if some $\rho_i$ is self dual, it will appear in both the sub $F^1 V^\rho \otimes
	\mathbb C$
	and the quotient $V^\rho \otimes \mathbb C/F^1 V^\rho \otimes
	\mathbb C$, and then the decomposition will not be isotypic.
\end{remark}

\begin{remark}
\label{remark:chevalley-weil-counterexample-constraint}
    Note that by the Chevalley-Weil formula
    theorem \cite[Theorem 2.1]{chevalleyW:uber}, as $H$-representations,
    $H^0(X, \omega_X) \simeq \on{triv} \oplus \on{reg}^{g-1} \oplus W,$ for
    $\on{triv}$ the trivial representation, and $\on{reg}$ the regular
    representation, where $W$ is an auxiliary representation which is $0$ if the
    cover is unramified, and is, in general, explicitly computable in terms of
    the ramification data of $f$.
    Therefore, conditions (1) and (2) of
    \autoref{proposition:representation-theoretic-criterion} each imply that $g\leq 1$.

    As pointed out by a referee, to verify the requirement $g \leq 1$, one only
    needs the easier fact that the
    $H$-module $H^1(X, \mathbb{Q})$ contains $\on{triv}^{\oplus 2} \oplus
    \on{reg}^{2g-2}$.
    This is somewhat easier to prove than Chevalley-Weil by considering an $H$-equivariant
    triangulation of $X$, and can be verified with a proof analogous to that of
    \cite[Proposition 1.1]{grunewald2015arithmetic}.
\end{remark}
\begin{remark}
	\label{remark:punctures-constraint}
	Adding on to \autoref{remark:chevalley-weil-counterexample-constraint}, the Chevalley Weil formula shows that for faithful representations we can only hope to apply
\autoref{proposition:representation-theoretic-criterion}(2) to families where $g =
1, n = 1$, as if there is more
than $1$ marked point (and the $H$-representation is nontrivial at $\geq 2$
points) or the genus is more than $1$, any such representation will appear with multiplicity greater than $1$.
\end{remark}

\begin{proof}[Proof of \autoref{proposition:representation-theoretic-criterion}]
Let $\mathscr H^\rho$ be the $\rho$-isotypic part of $\mathscr H := R^1 \pi'_* \mathbb Q
\otimes \mathscr O_{\mathscr M}$.
The connection on $\mathscr H$ induces a connection
$\nabla: \mathscr H^\rho \to \mathscr H^\rho \otimes \Omega^1_{\mathscr M}$.
In turn, this induces a $\mathscr O_{\mathscr M}$-linear map
$\overline{\nabla} : F^1 \mathscr H^\rho \to (\mathscr H^\rho/F^1 \mathscr
H^\rho) \otimes \Omega^1_{\mathscr M}$.
We claim that either of the conditions
in the theorem statement imply that $\overline{\nabla}$ is identically zero. By e.g.~\cite[(4.4.2)]{katz:pcurvature-and-hodge}, this implies $\mathscr H^\rho$
has finite monodromy, and so $\PW_{g,n}$ fails.

We now verify $\overline{\nabla}= 0$ in the two cases. Set $V^\rho=(\mathscr{H}^\rho)_m$. In case (1), $\overline
\nabla$ is an $H$-equivariant map
\begin{align*}
\oplus_{i \in S} \rho_i \otimes \hom(\rho_i, F^1 V^\rho) \to \oplus_{j
\notin S} \rho_{j} \otimes \hom(\rho_{j},
V^\rho/F^1 V^\rho) \otimes \Omega_{\mathscr M}
\end{align*}
between $H$-representations with no irreducible subrepresentations in common, and therefore must vanish identically.

To verify (2), 
observe that an irreducible symplectically self-dual $H$-representation $\pi$ satisfies
$(\sym^2 \pi)^H= 0$.
Indeed
$\hom(\on{triv}, \wedge^2\pi)$ is non-zero by assumption.
By Schur's lemma, $\hom(\on{triv}, \pi \otimes \pi)$ has dimension exactly $1$ (as $\pi$ is self-dual).
We then find
\begin{align*}
\hom(\on{triv}, \pi \otimes \pi) =
\hom(\on{triv}, \sym^2 \pi) \oplus \hom(\on{triv}, \wedge^2\pi)
\end{align*}
and so $\hom(\on{triv}, \sym^2 \pi) = 0$.

To check $\overline{\nabla}$ vanishes, we can write it as a sum of maps
$\overline{\nabla}_i : F^1 \mathscr H_{\rho_i} \to (\mathscr H_{\rho_i}/F^1 \mathscr
H_{\rho_i}) \otimes \Omega^1_{\mathscr M}$
and show each $\overline{\nabla}_i$ vanishes.
If the multiplicity of $\rho_i$ in $F^1 \mathscr H_{\rho_i}$ is $0$, we
have $F^1 \mathscr H_{\rho_i} = 0$, and so $\overline{\nabla}_i = 0$.
Otherwise, $\rho_i$ appears in $F^1 \mathscr H_{\rho_i}$ with multiplicity $1$.
For $\alpha$ as in
\eqref{equation:symmetric-map-repeated},
since $\rho_i$ appears with multiplicity $1$ in $H^0(X, \omega_X)$,
$\on{\im} \alpha \subset \sym^2 H^0(X, \omega_X)^H$ factors
through a copy of $\sym^2(\rho_1)^H = 0$. Hence $\alpha$ vanishes, so
$\overline{\nabla}_i$ vanishes.
\end{proof}
We now apply \autoref{proposition:representation-theoretic-criterion} to some
examples. We first consider the ``Eierlegende-Wollmilchsau'' example discussed in the paper of Putman
and Wieland where the Putman-Wieland conjecture was originally formulated.
\begin{example}[{\cite[Appendix A]{putmanW:abelian-quotients}}]
\label{example:quaternions}
Let $Q_8$ be the quaternion group (of order $8$). Putman and Wieland describe in
\cite[Appendix A]{putmanW:abelian-quotients} a family of $Q_8$ covers of genus one
curves $\mathscr{X}\to \mathscr{C}$, such that the vector space $V$ of
\autoref{proposition:representation-theoretic-criterion} is
$\mathbb{Q}$-irreducible as a $Q_8$-representation, with $F^1V_{\mathbb{C}}$
irreducible and symplectically self-dual. Hence \autoref{proposition:representation-theoretic-criterion}(2) applies.
\end{example}
\begin{example}[{\cite[Theorem 8.3]{mcmullen:braid-groups}}]
\label{example:genus-0}
The $8$ examples appearing in \cite[Table 10]{mcmullen:braid-groups} all satisfy
\autoref{proposition:representation-theoretic-criterion}(1), by the displayed
equation in the proof of \cite[Theorem 8.3]{mcmullen:braid-groups}. For example,
this includes a family of $\mathbb{Z}/4\mathbb{Z}$-covers of $\mathbb{P}^1$ which are totally ramified over $4$ points.
This family may be expressed in terms of equations as $y^4 = x(x-1)(x-\lambda)$,
for $\lambda \in \mathbb A^1_\lambda - \{0, 1\}$ the parameter of the family.
\end{example}
We can also make new examples using 
\autoref{proposition:representation-theoretic-criterion}.
\begin{example}
\label{example:2-dim-symplectic}
Let $H$ be any non-abelian finite subgroup of $\on{SL}_2(\mathbb C)$, with $\rho: H\hookrightarrow  \on{SL}_2(\mathbb C)$ the given embedding.
Suppose $H$ is generated by $2$
elements, say  $\gamma$ and $\delta$.
Then we claim there is a cover with Galois group $H$, furnishing a counterexample to $\PW_{1,1}$.

Indeed, since $H$ is non-abelian and generated by two elements, any two generators of $H$ cannot commute, 
so their commutator $c$ is nontrivial. Moreover $\rho$ is automatically
irreducible, as $H$ is non-abelian. Finally, note that $\rho$ is symplectically
self-dual, as the determinant yields a symplectic form preserved by $\rho$.
Since $\rho$ is two-dimensional
and semisimple, we obtain that $\rho(c)$ has two eigenvalues which are inverse
to each other and neither is equal to $1$.
Now, let $Y$ be a once punctured genus $1$ curve. Identify $\pi_1(Y, y) \simeq
F_2$, and let $X \to Y$ be the cover associated to the representation $F_2 \to H$
sending the first generator to $\gamma$ and the second
 to $\delta$.
It follows from the Chevalley-Weil formula, as in 
\cite[Theorem 2.1]{naeff:chevalley} that $\rho$ appears with multiplicity $1$ in
$H^0(X, \omega_X)$. Further, any Galois conjugate of $\rho$ is also faithful, and by
the same reasoning, it also appears with multiplicity $1$ in $H^0(X,
\omega_X)$.
Therefore, 
\autoref{proposition:representation-theoretic-criterion}(2)
applies and $f: X \to Y$ furnishes a counterexample to Putman-Wieland.

The non-abelian finite subgroups of $\on{SL}_2(\mathbb C)$ are the dicyclic groups and three exceptional groups
(the binary octahedral, binary icosahedral, and binary tetrahedral groups),
see
\cite{groupnames-2-dim-symplectic}.
One can directly check that each such group is generated by two elements.
Hence, for each such group, we can find a cover furnishing a counterexample to
Putman-Wieland with $g = 1$ and $n = 1$.

The cover of $\Sigma_{1,1}$ associated to an embedding of the order $4n$ dicyclic group
$\on{Dic}_n$ into $\on{SL}_2(\mathbb C)$ as above has genus $2n-1$.
By analyzing the representation theory of dicyclic groups,
one may verify the dimension of the fixed part is $n$ if $n$ is even and $n-1$ if $n$ is odd.
In particular, the dimension of this fixed part tends to $\infty$ as $n \to
\infty$.
\end{example}
Recall that a cover of a curve of genus one, branched over one point as in \autoref{example:2-dim-symplectic}, is called an \emph{origami curve}.

We can also make examples of representations of dimension more than $2$.
We expect there to be an unwieldy collection of these larger-dimensional
examples, and it is an enjoyable exercise to look for them.
\begin{example}
\label{example:4-dim-symplectic}
Let $H$ be the order $32$ group
given in terms of generators and relations by
\begin{equation}
H = \langle a,b,c | a^4=1, b^4=a^2, c^2=b\cdot a \cdot b^{-1}=a^{-1}, a\cdot c=c \cdot a, c \cdot b \cdot c^{-1}=a^{-1}b^3 \rangle
\end{equation}
As described at
\cite{groupnames-C4.10D4}, the commutator subgroup of $H$ is isomorphic to $(\mathbb Z/2\mathbb Z)^2$.
Since $c^2 = a^{-1}$, this group is generated by $b$ and $c$.
One may directly verify that $cbc^{-1}b^{-1}$ is not central. 
Let $Y$ be a once punctured genus $1$ curve. Upon identifying $\pi_1(Y, y) \simeq
F_2$, let $X \to Y$ be the cover associated to the representation $F_2 \to H$
sending the first generator to $c$ and the second to $b$.
There is a unique $4$-dimensional symplectic representation $\rho$ of $H$, and under
this representation, the character $\on{tr} \rho(cbc^{-1}b^{-1}) = 0$. 
Since $cbc^{-1}b^{-1}$ has order $2$, its eigenvalues are $\pm 1$, and as they sum to $0$, two must be $1$
and two must be $-1$.
The Chevalley-Weil formula as in \cite[Theorem 2.1]{naeff:chevalley}
implies $\rho$ appears with multiplicity $1$ in $H^0(X, \omega_X)$.
Since $\rho$ is the only symplectically self-dual representation of $H$, it has
no non-isomorphic Galois conjugates so we conclude $s = 1$ in
\autoref{proposition:representation-theoretic-criterion}.
We therefore conclude by 
\autoref{proposition:representation-theoretic-criterion}(2)
that $f: X \to Y$ furnishes a counterexample to Putman-Wieland, where here $g =
1$ and $n = 1$.
\end{example}

\subsection{Questions on families with large isotrivial isogeny factor}
\label{subsection:questions}

The examples \autoref{example:quaternions}-\autoref{example:4-dim-symplectic}
above seem interesting applications of
\autoref{proposition:1-dimensional-kernel}, but they only scratch the surface.
We were essentially able to characterize all $2$-dimensional counterexamples to
Putman-Wieland coming from
\autoref{proposition:representation-theoretic-criterion}(2) in
\autoref{example:2-dim-symplectic}, while we gave only a single $4$-dimensional
counterexample in \autoref{example:4-dim-symplectic} via a non-systematic
examination of $4$-dimensional symplectic representations.
Note such counterexamples can only occur when $g = 1$ and $n =1$, as remarked in
\autoref{remark:punctures-constraint}.
We suspect there are infinitely many more counterexamples obtainable in this
way, but do not have any systematic approach to searching for them.
\begin{question}
\label{question:}
Other than the example in \autoref{example:4-dim-symplectic}, are there more $4$-dimensional counterexamples to Putman-Wieland obtainable via
\autoref{proposition:representation-theoretic-criterion}?
Are there infinitely many such? Can one classify them?
\end{question}
Moreover, generalizing from $4$-dimensional representations to higher
dimensional representations, we ask the following.
\begin{question}
\label{question:}
Can one obtain families of origami curves with isotrivial isogeny factor coming
from \autoref{proposition:representation-theoretic-criterion}(2) via representations of arbitrarily large dimension? For each such dimension,
are there infinitely many? Can one classify them?
\end{question}
In \autoref{example:2-dim-symplectic}, we showed there are families of
origami curves of arbitrarily large genus $g$, whose Jacobians have an isotrivial isogeny factor of
dimension more than $g/2$.
\begin{question}
\label{question:proportion-bound}
What is the supremum over all positive real numbers $c$ so that, for arbitrarily large $g$,
there is a non-isotrivial family of genus $g$ curves whose Jacobians have an
isotrivial isogeny factor of dimension at least $c g$? What about the analogous question for families of origami curves?
\end{question}
As mentioned above, by \autoref{example:2-dim-symplectic}, $c \geq 1/2$, and so
$1/2 \leq c \leq 1$.

Even deciding whether $c < 1$ seems quite difficult, which leads to our next
question.
M\"oller showed in
\cite[Theorem 5.1]{moller:shimura-and-teichmuller-curves}, 
(using also \cite[Corollary 3.3]{moller:shimura-and-teichmuller-curves},)
that for $g \geq 6$,
there are no families of genus $g$ curves whose Jacobian has an isotrivial isogeny
factor of dimension $g - 1$.
More recently, 
Aulicino and Norton showed moreover that no such curves exist when $g = 5$ \cite{aulicinoN:shimura}.
\begin{question}
\label{question:codimension-bound}
For any fixed positive integer $m$, are there families of curves with an isotrivial isogeny
factor of dimension $g - m$ for arbitrarily large $g$?
\end{question}
As noted above, we find that the answer is ``no'' for $m = 1$.
If the answer to \autoref{question:codimension-bound} is ``yes" for any $m$, then
$c = 1$ in \autoref{question:proportion-bound}.

Finally, we note that many of the examples of isotrivial isogeny factors we
found above were isotypic in the following sense: 
consider a family of curves $\mathscr C \to B$
with a finite group $H$ acting on $\mathscr C$.
For any $b \in B$, the resulting $H$ action on $H^1(\mathscr C_b,
\mathscr O_{\mathscr C_b})$ 
decomposes as a sum of $\rho$-isotypic components
$H^1(\mathscr C_b, \mathscr O_{\mathscr C_b})^\rho$ over $H$-irreps $\rho$.
We say the isotrivial isogeny factor of the Jacobian of $\mathscr C \to B$ is
{\em isotypic} if,
for any $b \in B$, the tangent space to the maximal isotrivial isogeny factor is
identified with a direct sum of $\rho$-isotypic components of $H^1(\mathscr C_b,
\mathscr O_{\mathscr C_b})$.
\begin{question}
\label{question:}
Can one produce non-isotrivial families of curves $\mathscr C \to B$ with an
$H$
action as above whose Jacobian has an isotypic isotrivial isogeny factor and so that the quotient $\mathscr C/H$ has genus at least $2$?
Can one produce such families where $\mathscr C/H$ has arbitrarily large genus?
Can one produce such families of curves with an isotypic isotrivial
isogeny factor so that the base $B$ has large dimension?
\end{question}
As a partial answer to the last question above,
in a forthcoming paper, we will show it is impossible to have isotypic
isotrivial isogeny factors whenever the quotient curve
$\mathscr C/H$ has genus $h$ and dominates $\mathscr M_h$ if $h \geq 3$.
On the other hand, it would be quite interesting to find such an example when $h
= 2$. If such a family exists, it would give a counterexample to Putman-Wieland
in genus $2$.

\section{Counterexamples for hyperelliptic curves}
\label{section:counterexample-hyperelliptic}
We now explain that the Putman-Wieland conjecture fails if we restrict our attention to the hyperelliptic mapping class group, following Markovi\'{c} \cite{markovic}. In particular, this implies that the Putman-Wieland conjecture fails in genus $2$.

\begin{definition}
	\label{definition:}
	Choose a hyperelliptic involution $\iota$ acting on $S_g$ and let
	the {\em hyperelliptic mapping class
	group} $\HMod_{g} := \Mod_{g}^{\iota} \subset \Mod_{g}$
	denote the
	centralizer of $\iota$ in $\Mod_{g}$.
	Let $\HMod_{g,n} \subset \Mod_{g,n}$ denote the preimage of
	$\HMod_g \subset \Mod_g$ under the surjection $\Mod_{g,n} \to \Mod_g$.
\end{definition}
\begin{remark}
	\label{remark:}
	Any two such hyperelliptic involutions $\iota \in \Mod_{g}$ are
	conjugate by \cite[Proposition 7.15]{farbM:a-primer}.
\end{remark}
We next recall a foundational result on topology of the moduli
stack of hyperelliptic curves. 
For a reference, we recommend the survey article
\cite[Proposition 1]{gonzalezdiezH:moduli-of-riemann-surfaces-with-symmetry},
and the reader may also consult \cite[Theorem 2]{harvey:on-branch-loci}.

\begin{lemma}
\label{lemma:hyperelliptic}
	Let $g\geq 2, n \geq 0$ and let
	$\mathscr{H}_{g,n} \subset \mathscr{M}_{g,n}$ denote the closed substack
	parameterizing families of smooth proper hyperelliptic curves of genus $g$ with
	geometrically connected fibers and $n$ disjoint sections.
	Under
	the identification of $\Mod(\Sigma_{g,n})$ with $\pi_1(\mathscr
	M_{g,n})$, one may recover
	$\HMod_{g,n}$ as 
	the image of $\pi_1(\mathscr{H}_{g,n}) \to \pi_1(\mathscr
	M_{g,n})$.
	Moreover $\mathscr{H}_{g,n}$ is a $K(\pi_1(\mathscr H_{g,n}), 1)$.
\end{lemma}

\begin{remark}
	\label{remark:}
	The hyperelliptic mapping class group is sometimes also called the
	symmetric mapping class group
	\cite[\S9.4]{farbM:a-primer}.
\end{remark}
\begin{definition}
	\label{definition:}
	In the setup of \autoref{definition:pw},
	let $\HPW_{g,n}$ be the statement that for every finite index
characteristic subgroup $K \subset \pi_1(\Sigma_{g,n}, v_0)$, 
there are no nonzero vectors $v \in V_h := H_1(\Sigma_{g'}, \mathbb C)$
with finite orbit under the action of $\HMod_{g,n+1}$.
\end{definition}

\begin{problem}[Hyperelliptic Putman-Wieland problem, cf. \protect{\cite[Problem
	3.4]{boggi:hyperelliptic-arxiv}}]
	\label{problem:hyperelliptic-putman-wieland}
		Suppose $g \geq 2, n \geq 0$. Does $\HPW_{g,n}$ hold?
\end{problem}

The following result is essentially \cite[Theorem 1.3]{markovic} where the
case $g = 2$ is proven. We observe here that an analogous proof works for hyperelliptic curves in all genera.
Recall we use
$\mathscr{H}_{g,n}$ for the moduli stack of
    smooth hyperelliptic curves of genus $g$ with $n$ marked points, as in
    \autoref{lemma:hyperelliptic}. 
\begin{proposition}
    \label{proposition:hyperelliptic-counterexample}
    Let $g\geq 2, n \geq 0$ be integers.
    There exists a finite \'etale map
    $\mathscr{H}\to \mathscr{H}_{g,n}$, so that if $\mathscr C$ is the
    corresponding relative curve over $\mathscr H$,
    there is a finite \'etale cover $\mathscr{X}$ of $\mathscr{C}$ of degree 36 such that:
    \begin{enumerate}
        \item the composite map $\mathscr{X}\to \mathscr{C}\to \mathscr{H}$ has geometrically connected fibers.
	\item The relative Jacobian $\pic^0_{\mathscr X/\mathscr H}$ has an isotrivial isogeny factor.
    \end{enumerate}

    For $x\in \mathscr{H}$ a geometric point,
    the finite index subgroup of the hyperelliptic mapping class group $\pi_1(\mathscr{H},
    x)$
    has a vector with finite orbit under its action on $H^1(\mathscr X_x,
    \mathbb Q)$ and so the answer to the hyperelliptic Putman-Wieland problem is
    negative
    for every $g \geq 2,n \geq 0$.
    In particular, $\PW_{2,n}$ is false for all $n \geq 0$.
\end{proposition}
We will prove \autoref{proposition:hyperelliptic-counterexample} below in
\autoref{subsection:proof-hyperelliptic}.
We will use the following result of Bogomolov-Tschinkel:
\begin{theorem}[{\cite[Proposition 3.8, Remark 3.10]{BT:curve-correspondences}}]\label{theorem:BT-construction}
    Let $g\geq 2$ be an integer. 
    Let $C_0$ be the genus $2$ curve defined by $y^6=x(x-1)$ and $E_0$ be the
    genus $1$ curve defined by $y^3=x(x-1)$.
    Let $Z \in \{C_0, E_0\}$.
    There exists a finite \'etale cover
    $\mathscr{H}_Z$  of $\mathscr{H}_g$, and a finite
    \'etale cover
    $\mathscr{X}_Z$ of $\mathscr{C}_Z := \mathscr C_g \times_{\mathscr H_g} {\mathscr{H}_Z}$ such that the composite map
    $\pi'_Z: \mathscr{X}_Z\to \mathscr{C}_Z\to \mathscr{H}_Z$ has geometrically
    connected fibers, and such that every fiber of $\pi'_Z$ admits a non-constant
    map to the smooth projective curve $Z$.
	If $Z = C_0$, the cover $\mathscr X_Z \to \mathscr C_Z$ has degree $648$ 
	while if $Z = E_0$
	this map has degree $36$.
	The non-constant map $\mathscr X_Z \to Z \times \mathscr H_Z$ has degree
	$4$.
\end{theorem}
\begin{proof}
Strictly speaking, in 
\cite[Proposition 3.8, Remark 3.10]{BT:curve-correspondences}, it is only shown
that a hyperelliptic curve $C$ admits a finite \'etale cover $X$ with a non-constant map to $C_0$; however, their construction works in families, hence the statement. 
From their construction, the map is a composite of $5$ maps of respective
degrees $2, 9, 2, 9$,
and $2$, so has degree $2\cdot 9 \cdot 2\cdot 9 \cdot 2 = 648$.
(This was mistakenly claimed to be $72$ in \cite[Remark 3.10]{BT:curve-correspondences}
but updated in a later version  \cite[Remark 3.10]{BT:updated-curve-correspondences}.)
An intermediate step in the proof of
\cite[Proposition 3.8]{BT:curve-correspondences}
shows that the hyperelliptic curve $C$ admits a degree $36$ finite \'etale cover
$X$
with a degree $4$ non-constant map to $E_0$. 
This cover $X \to C$ is a composite of the first three maps of the
above-mentioned $5$ maps, which
have degrees $2, 9$, and $2$. 
Hence, this cover has degree $36$.
\end{proof}
\begin{remark} 
Note that there is a small typo in \cite[Example 3.7]{BT:curve-correspondences},
which claims that every hyperelliptic curve admits a non-constant map to an
elliptic curve satisfying some extra conditions; this claim is evidently false.
However, given a hyperelliptic curve $C$, one may construct a finite \'etale
cover $X$ of $C$ mapping to an elliptic curve and satisfying the desired
conditions as follows; such an $X$ suffices for the rest of their argument to go
through. Let $f: C\to \mathbb{P}^1$ be the hyperelliptic double cover, which we
may without loss of generality assume is branched over $0,1,\infty$. Let
$\lambda$ be another branch point of $f$, and let $E_\lambda$ be the elliptic
curve defined by $$E_\lambda: y^2=x(x-1)(x-\lambda).$$ Then any component $X$ of the normalized fiber product $\widetilde{E\times_{\mathbb{P}^1} C}$ satisfies the desired properties.
\end{remark}

In order to prove \autoref{proposition:hyperelliptic-counterexample}, we will
also use the following equivalences. This is a fairly standard Hodge-theoretic
argument, the key input being the theorem of the fixed part.
\begin{lemma}
       \label{lemma:trivial-equivalence}
       Suppose $\mathscr{M}'$ is a scheme and we have a relative smooth proper curve $\pi': \mathscr C' \to \mathscr{M}'$ with geometrically connected fibers.
       The following properties are equivalent.
\begin{enumerate}
       \item[(1)] For $x\in \mathscr{M}'$, the vector space $H_1(\mathscr{C}'_x,
               \mathbb{C})$ contains no non-zero vector with finite orbit under the action
               of $\pi_1(\mathscr{M}', x)$.
       \item[(2)] The local system $R^1 \pi'_* \mathbb C$ contains no
               sub-local system with finite monodromy.
       \item[(3)] The relative Jacobian $\pic^0_{\mathscr C'/\mathscr M'}$ has no isotrivial isogeny factor.
      \end{enumerate}
The following properties are also equivalent.
\begin{enumerate}
       \item[(1')] For $x\in \mathscr{M}'$, the vector space $H_1(\mathscr{C}'_x,
               \mathbb{C})$ contains no non-zero fixed vector under the action
               of $\pi_1(\mathscr{M}', x)$.
       \item[(2')] The local system $R^1 \pi'_* \mathbb C$ contains no
               trivial sub-local system.
       \item[(3')] The relative Jacobian $\pic^0_{\mathscr C'/\mathscr M'}$ has no constant isogeny factor.
\end{enumerate}
\end{lemma}
\begin{proof}
       The equivalence of (1), (2), and (3) follow from that of (1'), (2'),
       and
       (3')
       by passing to a suitable connected finite \'etale cover of $\mathscr{M}'$ upon which the
       vector with finite orbit in (1) becomes a fixed vector, the sub-local
       system with finite monodromy in (2) becomes trivial, and the
       isotrivial factor in (3) becomes trivial.

       It remains to verify the equivalence of (1'), (2'), and (3').
       Under the equivalence between $\mathbb C$ representations of $\pi_1(\mathscr{M}',x)$ and local
       systems of $\mathbb C$ vector spaces on $\mathscr{M}$, a non-zero fixed vector in
       $H_1(\mathscr{C}'_x, \mathbb{C})$ corresponds to a nontrivial sub-local
       system, which gives the equivalence between (1') and (2').
       Note here we are using Poincar\'e duality to identify $H_1(\mathscr{C}'_x, \mathbb{C})$
       and $H^1(\mathscr{C}'_x, \mathbb{C})$.

	We next verify $(3') \implies (2')$. 
	To this end, let
       $\phi: \pic^0_{\mathscr C'/\mathscr M'} \to \mathscr{M}'$ denote the structure map.
       A trivial isogeny factor
       $\iota: A \to \pic^0_{\mathscr C'/\mathscr M'}$, gives
       a trivial sub-local system 
	$R^1 (\phi \circ \iota)_* \mathbb C \subset R^1 \pi'_* \mathbb C$.

       We conclude by demonstrating $(2') \implies (3')$.
       Suppose we begin with a trivial sub-local system of 
       $\mathbb W \subset R^1 \phi_* \mathbb C$, 
       the theorem of the fixed part \cite[Corollaire 4.1.2]{deligne:hodge-ii}
       implies that $\mathbb W$ is also a variation of Hodge structure.
       Let $\mathbb W_{\mathbb Z} \subset \mathbb W$ denote the trivial $\mathbb Z$ local system on $\mathscr M$
       for which $\mathbb W_{\mathbb Z} \otimes_{\mathbb Z} \mathbb C \simeq
       \mathbb W$.
       We can then recover the trivial isogeny factor as the quotient
       $F^1(\mathbb W \otimes_{\mathbb C} \mathscr O_{\mathscr M'})^\vee / \mathbb W_{\mathbb Z}^\vee$.
       (On fibers, this corresponds to the fact that an abelian variety $A$ can
       be written as the quotient $H^0(A, \Omega_A)^\vee/ H^1(A, \mathbb Z)$.)
\end{proof}

\subsection{}
\label{subsection:proof-hyperelliptic}
\begin{proof}[Proof of \autoref{proposition:hyperelliptic-counterexample}]
Let $\mathscr{H} := \mathscr H_{E_0}, \mathscr{C} := \mathscr C_{E_0}, \mathscr{X} := \mathscr X_{E_0}$ be as
defined in \autoref{theorem:BT-construction}.
As shown there, the cover $\mathscr X \to \mathscr C$ has degree $36$.
The Jacobian of $E_0$ is an isogeny factor of $\pic^0_{\mathscr X/\mathscr H}$.

Having produced a relative Jacobian with a constant isogeny factor, the
equivalence of \autoref{lemma:trivial-equivalence}(1) and 
\autoref{lemma:trivial-equivalence}(3)
together with the identification of $\pi_1(\mathscr H_{g,n})$ with $\HMod_{g,n}$
from \autoref{lemma:hyperelliptic},
implies the answer to \autoref{problem:hyperelliptic-putman-wieland}
is negative for every $g \geq 2, n \geq 0$.

Finally,  $\PW_{2,n}$ fails because the natural map
$\mathscr{H}_2 \to \mathscr{M}_2$ is an isomorphism, and hence in genus $2$ the
hyperelliptic mapping class group is identified with the entire mapping class group
by \autoref{lemma:hyperelliptic}.
\end{proof}
\begin{remark}
\autoref{proposition:hyperelliptic-counterexample} was originally claimed in
\cite{boggi:hyperellptic}; unfortunately the proof there was incorrect (see
\cite[Remark 3.17]{boggi:hyperelliptic-arxiv}). We would like to thank Marco
Boggi for his extremely helpful email correspondence regarding this issue, which
was useful to our understanding of the Putman-Wieland conjecture.
The reader may also consult \cite[Corollary 6.11]{boggi:notes-on-hyperelliptic}.
\end{remark}
\begin{remark}\label{rmk:sharpness}
In some sense \autoref{proposition:hyperelliptic-counterexample} shows that our proof of \autoref{corollary:asymptotic-putman-wieland} and \autoref{thm:high-codim-version} is sharp.	Indeed, \autoref{theorem:BT-construction} shows that for each $g$ there exists a diagram $$\xymatrix{
\mathscr{X} \ar[r]^f \ar[rd]_{\pi'} & \mathscr{C} \ar[d]^\pi \\
& \mathscr{H}
}$$
with $\pi$ a relative curve of genus $g$, so that the induced map $\mathscr{H}\to \mathscr{M}_g$ dominates the hyperelliptic locus, and moreover such that $f$ has degree bounded independently of $g$, yielding a counterexample to  $\HPW_{g,0}$. The hyperelliptic locus has codimension $g-2$ in $\mathscr{M}_g$, so this does not contradict \autoref{thm:high-codim-version}; but that theorem shows that such an example is impossible with families of curves whose codimension $\delta_g$ in $\mathscr{M}_g$ grows such that $g-\delta_g\to \infty$ with $g$.
\end{remark}

\bibliographystyle{alpha}
\bibliography{bibliography-mcg-hodge-theory}

\end{document}